\documentclass[12pt]{amsart}

\usepackage{amssymb}
\usepackage{mathtools,xspace}

\usepackage{titletoc}
\usepackage[shortlabels]{enumitem}

\usepackage{mathrsfs}

\usepackage[usenames,dvipsnames]{xcolor}
\definecolor{darkblue}{rgb}{0.0, 0.0, 0.55}
\definecolor{bordeaux}{rgb}{0.34, 0.01, 0.1}

\usepackage[pagebackref,colorlinks,linkcolor=bordeaux,citecolor=darkblue,urlcolor=black,hypertexnames=true]{hyperref}

\renewcommand{\subset}{\subseteq}

\makeindex
\makeatletter
\newcommand{\mycontentsbox}{%
{
\parskip=-1.0pt
\newpage\printindex\tableofcontents}}
\def\enddoc@text{\ifx\@empty\@translators \else\@settranslators\fi
\ifx\@empty\addresses \else\@setaddresses\fi
\newpage\mycontentsbox
}
\makeatother

\contentsmargin{2.55em} 

 \numberwithin{equation}{section}

\makeatletter
\newsavebox\myboxA
\newsavebox\myboxB
\newlength\mylenA

\newcommand*\xoverline[2][0.75]{%
    \sbox{\myboxA}{$\m@th#2$}%
    \setbox\myboxB\null
    \ht\myboxB=\ht\myboxA%
    \dp\myboxB=\dp\myboxA%
    \wd\myboxB=#1\wd\myboxA
    \sbox\myboxB{$\m@th\overline{\copy\myboxB}$}
    \setlength\mylenA{\the\wd\myboxA}
    \addtolength\mylenA{-\the\wd\myboxB}%
    \ifdim\wd\myboxB<\wd\myboxA%
       \rlap{\hskip 0.5\mylenA\usebox\myboxB}{\usebox\myboxA}%
    \else
        \hskip -0.5\mylenA\rlap{\usebox\myboxA}{\hskip 0.5\mylenA\usebox\myboxB}%
    \fi}
\makeatother

\newtheorem{theorem}            {Theorem}[section]
\newtheorem{corollary}          [theorem]{Corollary}
\newtheorem{proposition}        [theorem]{Proposition}

\newtheorem{lemma}              [theorem]{Lemma}
\newtheorem{remark}         [theorem]{Remark}

\newtheorem{example}            [theorem]{Example}

\newcommand{\Gat}{\Gamma}
\newcommand{\gat}{\gamma}
\newcommand{\PhiG}{\Gat} 

\newcommand{\gv}{{\tt{g}}}
\newcommand{\hv}{{\tt{h}}}
\newcommand{\rv}{{\tt{r}}}

\def\cD{ {\mathcal D} }

\def\cF{ {\mathscr F} }

\def\cH{ {\mathcal H} }

\def\cK{ {\mathcal K} }

\def\cR{{ \mathcal R }}
\def\cS{{\mathcal S} }

\def\smatn{\mathbb S_n(\mathbb C)}

\def\gtupn{\mathbb S_n(\C)^\gv}
\def\gtupl{\mathbb S_\ell(\C)^\gv}

\def\bbS{{\mathbb S}}
\def\gtup{\bbS(\C)^\gv}
\def\rtup{\bbS(\C)^\rv}
\newcommand{\gtuphr}{\bbS(\C)^\rv}

\def\htup{\bbS(\C)^{\hv}}
\def\htupn{\bbS_n(\C)^{\hv}}

\def\C{\mathbb C}

\def\R{\mathbb R}

\newcommand{\onlyco}{\operatorname{\!-co}}
\newcommand{\opco}{\operatorname{opco}}
\newcommand{\oopco}{\operatorname{-opco}}
\newcommand{\matco}{\operatorname{matco}}

\def\gammamatco{\Gat\onlyco}
\def\gammaopco{\Gat\oopco}
\def\range{\operatorname{range}}
\def\cmatco{\overline{\matco}}
\def\sa{B(\cH)_{sa}^{\gv}}
\mathchardef\mhyphen="2D

\def\cH{\mathcal{H}}

\def\ax{\langle x \rangle}

\newcommand{\intr}{\operatorname{int}}

\newcommand{\KK}{\mathcal{K}}
\newcommand{\oK}{\mathfrak{K}}
\newcommand{\oKp}{\oK_{\rm mat}}
\newcommand{\oD}{\mathfrak{D}}

\newcommand{\PGat}{\Phi_\Gat}

\newcommand{\KKn}{\KK(n)}
\newcommand{\TV}{\operatorname{TV}}

\newcommand{\woD}{\widehat{\oD}}
\newcommand{\wcD}{\widehat{\cD}}

\newcommand{\spos}{\mathcal{P}}
\newcommand{\sopos}{\mathfrak{P}}

\newcommand{\limsot}{\operatorname{SOT-lim}}
\newcommand{\limwot}{\operatorname{WOT-lim}}

\newcommand{\Coup}{\mathcal{C}_\Gamma}
\newcommand{\oCoup}{\mathfrak{C}_\Gamma}

\newcommand{\df}[1]{{\bf{#1}}{\index{#1}}}

\newcommand{\indifferent}{regular\xspace}

\title[$\Gat$-convex sets]{Noncommutative partial convexity via $\Gat$-convexity}

\author[M.T. Jury]{Michael Jury${}^1$}
\address{Michael Jury, Department of Mathematics\\
  University of Florida\\ Gainesville }
\email{mjury@ufl.edu}
\thanks{${}^1$Research Supported by NSF grant DMS-1900364}

\author[I. Klep]{Igor Klep${}^{2}$}
\address{Igor Klep, Department of Mathematics, 
University of Ljubljana, Slovenia}
\email{igor.klep@fmf.uni-lj.si}
\thanks{${}^2$Supported by the 
Slovenian Research Agency grants J1-8132, N1-0057 and P1-0222. Partially supported 
by the 
Marsden Fund Council of the Royal Society of New Zealand.}

\author[M.E. Mancuso]{Mark E. Mancuso${}^3$}
\address{Mark E. Mancuso, Department of Mathematics and Statistics, 
  Washington University in St. Louis
  }
\email{mark.mancuso@wustl.edu}
\thanks{${}^3$Research partially supported by NSF grant DMS-1565243}

\author[S. McCullough]{Scott McCullough${}^4$}
\address{Scott McCullough, Department of Mathematics\\
  University of Florida\\ Gainesville 
   }
   \email{sam@math.ufl.edu}
\thanks{${}^4$Research supported by  NSF grants DMS-361501 and DMS-1764231}

\author[J.E. Pascoe]{James Eldred Pascoe${}^5$}
\address{James Pascoe, Department of Mathematics\\
  University of Florida\\ Gainesville 
   }
   \email{pascoej.ufl.edu}
\thanks{${}^5$Partially supported by NSF MSPRF DMS 1606260.}



\numberwithin{equation}{section}

\makeindex

\begin{document}

\begin{abstract}
Motivated by classical notions of
partial convexity, biconvexity, and bilinear matrix inequalities,
we investigate the theory of free sets that are defined by (low degree) noncommutative
 matrix polynomials with constrained terms. Given a tuple of symmetric polynomials $\Gat$,
a free set $\cK$ is called $\Gat$-convex if
for all
 $X\in\cK$ and isometries $V$ satisfying
$V^*\Gat(X)V=\Gat(V^*XV)$, we have $V^*XV\in \cK.$
We establish an Effros-Winkler Hahn-Banach separation theorem for $\Gat$-convex sets; 
they are delineated by  linear pencils in the coordinates of $\Gat$ and the variables $x$.
\end{abstract}

\maketitle

\thispagestyle{empty}

\section{Introduction}
Convexity is ubiquitous in quantitative sciences. A set 
$C\subset \mathbb R^\gv$ is convex if for any two points in $C$ the line segment connecting 
them lies entirely in $C$. Such sets, whenever they are closed, are described by (possibly infinitely many) affine
linear inequalities. Convexity is fundamental in many areas of mathematics, including functional 
analysis, optimization, and geometry \cite{Bar}. The convex sets described 
by finitely many linear inequalities 
are precisely the  polytopes, a very restrictive class. A much bigger, but 
still very tractable class of convex sets $C$,
which are the central objects in semidefinite programming \cite{BPT},
 are described by linear matrix inequalities (LMIs), i.e.,
\[
C=\{x\in\R^\gv\mid A_0+\sum_j A_jx_j\succeq0\},
\]
where $A_j$ are self-adjoint $d\times d$ matrices and 
$T\succeq0$ means that the self-adjoint matrix $T$ is positive semidefinite. Such sets are called 
spectrahedra. They appear in several branches of mathematics, 
e.g. optimization and algebraic geometry \cite{BPT}. 

The \df{linear pencil} $L(x)=A_0+\sum_j A_jx_j$ is naturally evaluated at tuples of self-adjoint $n\times n$ matrices $X$ by
\[
L(X)=A_0\otimes I_n + \sum_j A_j\otimes X_j
\]
leading to the notion of a \df{free spectrahedron} 
\[
\cD_L=  (\cD_L(n))_n
\quad\text{where }\ \cD_L(n)=\{X\in \smatn^\gv \mid L(X)\succeq0\},
\]
and $\smatn$ denotes the set of all $n\times n$ self-adjoint matrices. \index{$\smatn$}
Free spectrahedra arise naturally in applications such as systems engineering  and control theory \cite{boyd}.
They are matrix convex sets \cite{EW,CE,FHL,HKM13,HKMjems,HL,DDSS,zalar} 
 and are dual to operator systems and thus
intimately connected to the theory of completely positive maps \cite{paulsen}.
Moreover, the  Effros-Winkler  Hahn-Banach
 separation theorem \cite{EW}
 says that matrix convex sets  are determined
 by LMIs in much the same way that convex sets are determined by linear inequalities.

Sets and functions
 that have some partial convexity or other geometric features,
 say convex in some coordinates with the other held fixed, 
 arise in applications.  
 Free noncommutative polynomials,
 and more generally free rational functions, arise in
  engineering systems problems governed by a signal flow diagram.
  Typically, there are two classes
 of variables. The system variables depend on the choice of system parameters
 and produce polynomial, or more generally rational,  inequalities
 (in the sense of positive semidefinite) in the state variables. The algebraic
 form of these inequalities involves (matrix-valued) free polynomials
 or rational functions and depends only upon the flow diagram,
 and not the particular choice of system variables.  Convexity in the state
 variables, for a given choice of system variables, 
 is an important optimization consideration. One way to study partial convexity is through 
 bilinear matrix inequalities (BMIs) \cite{KSVS04}\footnote{See also  the MATLAB toolbox, {\url{https://set.kuleuven.be/optec/
 Software/bmisolver-a-matlab-package-for-solving-optimization-problems-with-bmi-constraints.}}}.
 A BMI is an expression of the form
\begin{equation}
\label{e:BMI}
A_0 + \sum A_j x_j + \sum_k B_k y_k + \sum_{p,q} C_{pq} x_p y_q \succeq0
\end{equation}
for self-adjoint matrices $A_j,B_k,C_{p,q}$ \cite{VAB}. 
Domains defined by BMIs are convex in the $x$ and $y$ variables separately.
The article \cite{HHLM} contains some  noncommutative results on partial convexity.

In analogy with matrix convexity and BMIs, 
it makes sense to consider matrix polynomial inequalities
built from a restricted set of predetermined polynomials
giving rise to the notion of $\Gat$-convexity.
One type of inequality we consider is of the form
\[
I-Ax-By-C y^2\succeq0.
\]
Sets describable in this form are ``convex in $x$'' and unconstrained in $y$.
That is, by allowing extra nonlinear terms in the 
matrix inequality, we can isolate certain geometric features of the domain. Allowing $xy,yx$
 as in \eqref{e:BMI},  obtains a class of biconvex sets.

Often times in this setting, the results, while finite-dimensional in nature, require working with operator inputs 
or coefficients for the inequalities, fitting in with larger trends in 
matrix convexity  \cite{DK,EE18,EE+,PS,PSS}  and the emerging area of free 
analysis  \cite{AM, BMV, KVV, PV, Pop, SSS}.

\subsection{Free polynomials and their evaluations}
\label{sec:basicfree}
Let $x=(x_1,\dots,x_\gv)$ denote a $\gv$-tuple of freely noncommuting variables.
Let $\ax$  denote the semigroup of  words in $x$ and we often use \index{word}
 $1$ to denote the unit $\varnothing.$ Let $\C\ax$ the algebra of finite $\C$-linear
combinations of words in $x$. \index{$\ax$} \index{$\C\ax$} 
An element $p\in \C\ax$ is a \df{free polynomial}, or just polynomial for short, and
takes the form
\begin{equation}
\label{e:defp}
 p(x) =\sum_{w\in\ax}  p_w w,
\end{equation}
where the sum is finite and $p_w\in \C$. \index{involution ${}^*$}
There is a natural involution ${}^*$ on $\ax$ determined by $x_j^*=x_j$ for $1\le j\le \gv$
and $(wu)^* = u^* w^*$ for $u,w\in\ax$.   This involution naturally extends to $\C\ax$.
For instance, for the polynomial $p$ of equation \eqref{e:defp}, 
\[
 p^* = \sum \overline{p_w} w^*.
\]
Since $x_j^*=x_j$ the variables $x$ are referred to as \df{symmetric variables}.

Let $\gtup$ denote the sequence, or graded set, \index{$\gtup$} \index{$\gtupn$}
$(\gtupn)_{n=1}^{\infty}$, where $\gtupn$ is the set of $\gv$-tuples $(X_1,\dots,X_\gv)$
of selfadjoint elements of $M_n(\C)$. 
 Elements of $\C\ax$ are naturally  \df{evaluated} at an $X\in\gtup$. For a word 
\[
 w = x_{j_1}\, x_{j_2} \, \cdots \, x_{j_N} \in \ax,
\]
and $X\in \gtupn$, 
\[
 w(X) = X_{j_1}\, X_{j_2}\, \cdots\, X_{j_N} \in M_n(\C).
\]
Given $p$ as in equation \eqref{e:defp}, \index{$p(X)$}
\[
 p(X) =\sum_{w\in\ax} p_w w(X) \in M_n(\C).  
\]
Thus $p$ determines a (graded)  function $p:\gtup \to M(\C)$,
where $M(\C)$ is the (graded) set $M(\C)=(M_n(\C))_n$.  Likewise
a tuple $p=(p_1,\dots,p_\rv)\in \C\ax^{1\times \rv} = M_{1,\rv}(\C\ax)$ determines
a mapping $p:\gtup\to M(\C)^\rv$.  \index{$M(\C)$}  \index{mapping}

A polynomial is \df{symmetric} if
it is invariant under the involution. As is well-known, $p\in \C\ax$
 is symmetric if and only if  $p(X)^* = p(X)$ for all $X\in\gtup$.
 In this case $p$ determines a mapping $p:\gtup \to \mathbb S(\C)^1$.

  Certain matrix-valued free polynomials will play an important role
 in this article. 
 A $\mu\times\mu$ matrix-valued free polynomial $p\in M_\mu(\C\ax)$  takes the form 
 of equation \eqref{e:defp}, but now  the coefficients $p_w$ lie in $M_\mu(\C).$
  Such a polynomial
$p$ is \df{evaluated} at a tuple $X\in \gtupn$ using the tensor (Kronecker) product as 
\[
 p(X) = \sum p_w\otimes w(X) \in M_\mu(\C)\otimes M_n(\C),
\]
 and $p$ is \df{symmetric} if  $p(X)^* = p(X)$ for all $X\in \gtup$. Equivalently,
$p$ is symmetric if $p_{w^*}=p_w^*$ for all words $w$.

\subsection{$\Gat$-convex sets}
\label{sec:gat-convex-sets}
 Let $\Gat=(\gat_1,\dots,\gat_\rv)$ denote a tuple  of {\it symmetric} free polynomials
 with $\gat_j=x_j$ for $1\le j\le \gv\le \rv$. We also use  $\PhiG:\gtup \to \rtup$
 to  denote the
 resulting mapping, \index{$\PGat$} \index{$\Gat$}
\begin{equation*}
 \PhiG(X) = (\gat_1(X),\dots,\gat_\rv(X)).
\end{equation*}
 A pair $(X,V)$, where $X\in \gtupn$
and $V:\C^m\to \C^n$ is an isometry, is a \df{$\Gat$-pair} provided 
\[
 V^* \PhiG(X)V= \PhiG(V^*XV).
\]
 Let $\Coup$ denote the collection of $\PhiG$-pairs. \index{$\Coup$}
 For instance, if $U$ is an $n\times n$ unitary matrix
 and  $X\in \gtupn$, then $(X,U)$ is a $\Gat$-pair.

A subset $\KK\subset \gtup$ is a sequence  $(\KKn)_n$ where
$\KKn\subset \gtupn$ for each $n$. A subset $\KK\subset \gtup$ is a
\df{free set} if it is closed with respect to direct sums,
 simultaneous  unitary similarity, and restrictions to reducing 
subspaces.\footnote{
Explicitly, if $X\in \KK(n)$ and $Y\in \KK(m)$, then $X\oplus Y\in\KK(n+m);$
if $U$ is a $n\times n$ unitary, then $U^*XU=(U^*X_1 U,\dots,U^*X_\gv U)\in \KK(n)$;
and if $\cK \subset \C^n$ is $k$ dimensional  reducing subspace for $X$, 
then  $X|_{\cK}\in \KK(k).$}
A set $\KK$ is a \df{$\Gat$-convex set} if it is free and if 
\[
 X\in \KK \, \text{ and } \, (X,V)\in \Coup \, \implies V^*XV\in \KK.
\]
In the special case that $\rv=\gv$ (equivalently $\Gat(x)=x$) 
$\Gat$-convexity reduces to ordinary 
 matrix convexity.

\begin{example}\rm
 \label{eg:TV}  
  Consider the case of two variables $(x,y)=(x_1,x_2)$
 and $\Gat=\{x,y,y^2\}$.  For notational ease, we write
 $y^2$-convex instead of $\{x,y,y^2\}$-convex. 
 A pair $((X,Y),V)$ is in $\Coup$ if and only if
 the range of $V$ reduces $Y$ and, as shown in Proposition 
 \ref{prop:y2convex=convexinx}, a free set $\KK$ is $y^2$-convex
 if and only if $(X_1,Y),(X_2,Y)\in \KKn$ implies $(\frac{X_1+X_2}{2},Y)\in\KKn$.
 Using either of these criteria, it is readily verified that, for $d$
 a positive integer, the TV screen  $\TV^d=(\TV^d(n))_n$ defined by
\begin{equation}
\label{e:TV}
 \TV^d(n) = \{(X,Y): I-X^2-Y^{2d} \succeq 0\} \subset {\mathbb S_n(\C)^2}
\end{equation}
 is a free set that is $y^2$-convex. 
 Section \ref{sec:y2convex} treats 
 $y^2$-convexity. In these directions, see  also  \cite{HHLM,BM,DHM}.
\qed
\end{example}

\begin{example}\rm
\label{eg:xyconvex}
 Let
 $\Gat = \{x,y,xy+yx, i(xy-yx))\}.$  It is straightforward to 
 verify that $((X,Y),V)$ is a $\Gat$-pair if and only if 
 $V^* XY V = V^* X V\, V^* YV$. Thus it is sensible, for notational purposes, to write
  $xy$ in place of the more cumbersome $(xy+yx, i(xy-yx))$ and
  $xy$-convex in place of $\{x,y,xy+yx, i(xy-yx)\}$-convex.   The convexity in this example
  is intimately connected with Bilinear Matrix Inequalities (BMIs)\footnote{See for instance
 the MATLAB toolbox, {\url{https://set.kuleuven.be/optec/
 Software/bmisolver-a-matlab-package-for-solving-optimization-problems-with-bmi-constraints.}}}
  as explained in Subsection \ref{sec:HBandP},
which were previously studied in \cite{KSVS04}; see also \cite{HHLM} for some  noncommutative results.\index{Bilinear Matrix Inequalities}\index{BMI}
\qed
\end{example}

 Of course a theory of convexity should contain Hahn-Banach separation results.
 In the case of matrix convex sets, this role is played by monic linear pencils and the
 \index{Effros-Winkler Theorem}
  Effros-Winkler Matricial Hahn-Banach separation theorem \cite{EW},
 or just the \df{Effros-Winkler theorem} for short.  A \df{monic linear pencil}
 $M\in M_\mu(\C\ax)$ is a symmetric matrix-valued polynomial of the form
\[
 M(x) = I_\mu +\sum_1^\gv A_j x_j,
\]
 where $A\in \bbS_\mu(\C)^\gv.$  We refer to $\mu$ as the \df{size of the pencil} $M.$
 The following version of the Effros-Winkler theorem can be found in \cite{HMannals,HKMjems}.
 
\begin{theorem}
 \label{t:EW}
  If $\KK\subset \gtup$ is a closed matrix convex set containing the origin and if 
  $Y\notin\KK$, then there is a monic linear pencil $M$ such that $M(\KK)\succeq 0,$
  but $M(Y)\not\succeq 0$. Furthermore, if $Y$ has size $\ell$, then $M$ can be chosen 
  to have size $\ell.$
\end{theorem}

  For $\Gat$-convex sets, the analog of a monic linear pencil
  is a \df{monic $\Gat$-pencil} -- a symmetric
 $L\in M_\mu(\C\ax)$ of the form
\[
 L(x) = I_\mu +  \sum_{j=1}^\rv A_j \gat_j(x),
\]
 where $A\in \bbS_\mu(\C)^\rv.$  
 In Section \ref{sec:generalstuff}
 analogs of the Effros-Winkler theorem for $\Gat$-convex sets are
 established. For instance, 
 a basic separation result that follows from Theorem \ref{thm:jp} is the following.

\begin{theorem}
 \label{t:thm:jp-intro}
   Suppose $\KK\subset \gtup$ is closed, $\Gat$-convex, contains $0,$ and $\Gamma(0)=0.$
  If the matrix
 convex hull of $\Gat(\KK) \subset \gtuphr$ is closed and if 
 $Y\in \mathbb S_\ell(\C)^\gv \setminus \KK(\ell)$, then there is a monic
 $\Gat$-pencil $L$ of size $\ell$ such that $L$ is positive semidefinite on $\KK$, but
 $L(Y)$ is not positive semidefinite.
\end{theorem}

 Operator convexity is defined at the outset of Section \ref{sec:opset}. In
 the operator setting, the closedness hypothesis on the (operator) convex hull of
 $\Gamma(\KK)$ is not needed. It is closed automatically by 
 Theorem \ref{thm: sotclosed}.

\begin{theorem}[Theorem~\ref{thm:GatOPEW}]
 \label{t:GatOPEW-intro}
  Suppose $\oK$ is a bounded, strong operator topology closed, $\Gat$-convex set
  that contains $0$ and $\Gamma(0)=0.$ If $Y\notin\oK$, then there exists
  a positive integer $N$ and a monic $\Gat$-pencil $L$ of size $N$ such that $L$ takes
  positive semidefinite values on $\oK$, but $L(Y)$ is not positive semidefinite.
\end{theorem}

By direct summing all $\Gat$-pencils 
with rational coefficients
that are positive semidefinite on $\oK$, one obtains a single operator $\Gat$-pencil $L$ with 
bounded coefficients whenever $0$ is in the
interior of the convex hull of $\Gamma(\oK).$
That is, by this routine argument we obtain the following corollary.

\begin{corollary}
\label{t:OneOpPencil-intro}
Suppose $\oK$ is strong operator topology closed, $\Gat$-convex, bounded, 
contains $0$, and $\Gamma(0)=0.$ Suppose also that $0$ is in the
interior of the convex hull of $\Gamma(\oK).$
Then there exists a monic operator $\Gat$-pencil $L$ such that 
$\oK=\{X: L(X)\succeq 0\}.$
\end{corollary}

The hypotheses of Corollary \ref{t:OneOpPencil-intro}
are met whenever the span of $\Gamma$ does not contain a positive polynomial, which is the content of
Theorem \ref{t:pospexists}.
 For example, this holds whenever 
the coordinates of $\Gamma$ are multilinear.

 \subsection{Reader's guide}
 Section \ref{sec:generalstuff} develops the framework
 of $\Gat$-convexity. Subsection \ref{ssec:co}
 discusses $\Gat$-convex hulls and proves an analog
 of the Effros-Winkler theorem for $\Gat$-convex sets in 
 Theorem \ref{thm:jp}.
 Free semialgebraic sets and $\Gat$-convex polynomials
 are introduced and treated in Subsection \ref{ssec:bicvxPoly}.
 Section \ref{sec:opset} 
 deals with a more robust class of $\Gat$-convex sets.
Here, by adding extra points, further structure is obtained.
The article concludes with an investigation of
$y^2$-convexity in Section \ref{sec:y2convex}.

The authors thank Bill Helton for his insights and  helpful conversations.

\section{General Theory of $\Gat$-Convex Sets} 
\label{sec:generalstuff}

We now introduce the basic notions related to $\Gat$-convexity. The main result of this section
is Theorem \ref{thm:jp}, giving an Effros-Winkler type separation result for $\Gat$-convex sets. 
In Subsection \ref{ssec:bicvxPoly},
we touch upon $\Gat$-convex polynomials, 
a topic we explore further in the accompanying paper
\cite{JKMMP+}.

\subsection{Convex hulls}\label{ssec:co}
 The \df{$\Gat$-convex hull} of a free set $\KK$,
 the intersection of all $\Gat$-convex sets containing $\KK$,
 is denoted $\gammamatco(\KK).$ \index{$\gammamatco(\KK)$} 
 When  $\rv=\gv$ (and thus $\Gat(x)=x$) $\gammamatco(\KK)$
 is the ordinary matrix convex hull of $\KK$, denoted $\matco(\KK)$.
 \index{$\matco(\KK)$}

\begin{proposition}
\label{prop:gammahull}
 If $\KK \subset \gtup$ is a free set, then 
\[
 \gammamatco(\KK) =\{V^*XV: X\in \KK, \, (X,V)\in\Coup \}.
\] 
\end{proposition}

\begin{proof} 
 Let $\mathscr{K} =\{V^*XV: X\in \KK, \, (X,V)\in\Coup\}.$ It is readily verified that
$\mathscr{K}$ is a $\Gat$-convex set that contains $\KK$. On the other hand, 
 by definition, 
$\gammamatco(\KK)$ must contain $\mathscr{K}$. 
\end{proof}

\begin{proposition}
\label{prop:jp}
Suppose $\KK\subset \gtup$ is a free set and $X\in \gtup$.  The point 
$X$ is in $\gammamatco(\KK)$ if and only if 
$\PhiG(X)$ is in $\matco(\PhiG(\KK)).$ Equivalently,
\[
 \PhiG^{-1}(\matco(\PhiG(\KK))) 
   = \gammamatco(\KK).
\]
\end{proposition}

\begin{proof}
First suppose $X\in \gammamatco(\KK)$. By Proposition \ref{prop:gammahull},  there 
exists a $Y$ in $\KK$ and an isometry $V$ such that 
$V^*\Gat(Y)V= \Gat(V^*YV)$ and $X=V^*YV$. Thus,  $\PhiG(X)=V^*\PhiG(Y)V,$ and therefore
$\PhiG(X) \in \matco(\PhiG(\KK)).$

Conversely, suppose $\PhiG(X)\in \matco(\PhiG(\KK))$. There is a $Y\in \KK$  and 
an isometry $V$ such that $\PhiG(X)  = V^*\PhiG(Y)V.$ 
Comparing the first $\gv$  coordinates gives $X=V^*YV$ and  hence $(Y,V)$ is $\Gat$-pair.
 Since $Y\in \KK$ and $(Y,V)$ is a $\Gat$-pair,  
 Proposition \ref{prop:gammahull} implies  $X\in \gammamatco(\KK)$.
\end{proof}

\begin{proposition}
 The projection of $\matco(\PhiG(\KK))$ onto the first $\gv$ coordinates is $\matco(\KK)$.
\end{proposition}

\begin{proof}
 The set $\matco(\PhiG(\KK))$ is matrix convex. Hence its projection onto the first 
$\gv$ coordinates is matrix convex and contains $\KK.$ Therefore this projection contains 
$\matco(\KK)$. On the other hand, by definition, this projection must be contained in $\matco(\KK)$.
\end{proof}

\subsection{Hahn-Banach separation and pencils}
\label{sec:HBandP}
 \index{$M_\mu(\C\ax)$} \index{$p(X)$}

As a special case of a matrix-valued free polynomial,
 a \df{$\Gat$-pencil (of size $d$)} 
is  a  (affine) linear pencil $L$ in $(\gat_1,\dots,\gat_\rv)$ whose
coefficients lie in $\bbS_d(\C).$ Thus, there is a 
positive integer $d$ and $A_0,A_1,\dots,A_\rv\in \bbS_d(\C)$ such that 
\begin{equation}
\label{e:Gat-pencil}
 L(x) = A_0 +\sum A_j \gamma_j(x).
\end{equation}
Since $A_j\in \bbS_d(\C)$ and $\gamma_j$ are symmetric polynomials, $L$
 is symmetric, and hence at times we refer to $L$ as a symmetric $\Gat$-pencil.
The pencil $L$ is \df{monic} if $A_0=I_d.$
For example, in the case of two variables $(x,y)$, 
a symmetric monic $xy$-pencil can be expressed as
\[
 L(x,y) = I + A_x x + A_y y + Bxy + B^*yx,
\]
where $A_x,A_y$ are self-adjoint. A symmetric monic $y^2$-pencil is of the form
\[
L(x,y) = I +A_x x + A_y y + By^2,
\] where $A_x, A_y$, and $B$ are self-adjoint. In the special case that $\rv=\gv$ (equivalently
 $\Gat(x)=x)$,
$L(x)= I+\sum A_j x_j$ is known as a \df{monic linear pencil}.

A pencil $L$ with coefficients in $\mathbb S_d(\C)$
 as in equation \eqref{e:Gat-pencil} is {\bf evaluated}  at a
tuple $X\in \gtupn$ using the tensor  product as
\[
 L(X) = A_0\otimes I_n + \sum A_j\otimes \gamma_j(X) \in 
  \bbS_d(\C)\otimes \bbS_n(\C).
\]
 The \df{free semialgebraic sets} associated to a symmetric $\Gat$-pencil $L$, 
\[
 \wcD_L :=\{X \in\gtup: L(X)\succeq 0\} \ \ (\text{resp.} \, \spos_L :=\{X\in\gtup: L(X)\succ 0\}),
\]
are $\Gat$-convex.  A difficult question is to determine when a closed (resp. open)
$\Gat$-convex set is the positivity (resp. strict positivity) set of
a $\Gat$-linear pencil.

\index{$\cmatco(S)$}
Given a subset $\cS\subset \gtup$, its (levelwise) closed matrix convex hull is 
  denoted by $\cmatco(\cS).$
Thus  $\cmatco(\cS)(n)$ is the closure of $\matco(\cS)(n)$ in $\gtupn$.
A routine argument shows $\cmatco(\cS)$ is also matrix convex.

\begin{theorem}
\label{thm:jp}
  Suppose $\KK$ is $\Gat$-convex, contains $0$, and $\PhiG(0)=0$. 
 If $Y\in \gtupl$ and $\PhiG(Y)\notin \cmatco(\PhiG(\KK))(\ell),$  then there exists a 
 monic $\Gat$-pencil $L$ of size $\ell$ such that $L(\KK)\succ 0$, but $L(Y)\not\succeq 0$. 
 In particular, if $\matco(\Gat(\KK))$ is closed, then for each $Y\notin \KK(\ell)$ there exists
 a monic $\Gat$-pencil $L$ of size $\ell$ such that
 $L(\KK)\succ 0$, but $L(Y)\not\succeq 0$.
\end{theorem}

\begin{proof}
 Since $\PhiG(Y)\notin\cmatco(\PhiG(\KK))(\ell)$ and $\cmatco(\PhiG(\KK))$ is 
 a closed matrix convex subset of $\rtup$ containing $0,$
 Theorem \ref{t:EW} 
implies there is a monic linear pencil 
\[
M(z) = I_\ell +\sum_{j=1}^\rv A_jz_j
\]
 of size $\ell$  such that $M(Z)\succeq 0$ for all 
$Z\in \cmatco(\PhiG(\KK))$, but
 $M(\PhiG(Y))\not\succeq 0$. Thus  $L^\prime = M\circ \PhiG$ is  a monic $\Gat$-pencil
of size $\ell$ that 
is indefinite at $Y$ and positive semidefinite on $\KK.$ 
 Replacing $L^\prime$ by 
\[
  L(z)=tI+(1-t)L^\prime(z)  =  I_\ell + \sum tA_j z_j
\]
 for small enough
 $t\in (0,1)$ produces a monic $\Gat$-pencil of size $\ell$ indefinite at $Y$ such that 
$L(\KK)\succ 0.$

To complete the proof, suppose $\matco(\PhiG(\KK))$ is closed and 
 $Y\notin \KK =\gammamatco(\KK)$.
 Since, by Proposition \ref{prop:jp}, $\PhiG(Y)\notin \matco(\PhiG(\KK))=\cmatco(\PhiG(\KK))$
the existence of $L$ follows from what has already been proved.
\end{proof}

\begin{remark}\rm
In Theorem \ref{thm:jp}, $\matco(\PhiG(\KK))$ can be replaced by any matrix convex 
set $\cR$ containing $\matco(\PhiG(\KK))$ such that
\[
 \gammamatco(\KK) =\PhiG^{-1}(\cR\cap \range(\PhiG)).
\]
This ambiguity complicates the problem of determining when a $\Gat$-convex
set is the positivity set of a monic $\Gat$-pencil. 
\qed \end{remark}

\begin{theorem}
 \label{t:pospexists}
   Suppose $\KK$ is $\Gat$-convex, contains $0$ and  $\Gat(0)=0.$ 
   The real span of $\{\gamma_j:1\le j\le \rv\}$  
   contains a polynomial $q\in\C\ax$ such that $q(X)\succeq 0$ for $X\in \KK$ if and only if
   $0$ is not in the interior of $\matco(\Gat(\KK))(1)\subset \mathbb R^\rv$.
\end{theorem}

\begin{proof}
 Suppose $0$ is not in the interior of $\matco(\Gat(\KK))(1)$. In this 
 case $0$ is in the boundary of $\matco(\Gat(\KK))(1)$, since $\Gat(0)=0.$
 Hence, as $\matco(\Gat(\KK))(1)$ is convex, there exists a linear functional
 $\lambda:\R^\rv \to \R$ such that $\lambda$ is nonnegative on 
 $\matco(\Gat(\KK))(1).$ Thus $\lambda(z) = \sum_{j=1}^\rv \lambda_j z_j$ for
 some $\lambda_j\in\mathbb R.$ 
  Set $q=\sum_{j=1}^\rv \lambda_j \gat_j.$
  
 Suppose $n$ is a positive integer, $Y\in \gtupn$ and
 $h\in\C^n$ is a unit vector. Identify $h$ as an isometry $h:\C\to\C^n$,
 let $y= h^* \Gat(Y) h$ and observe,
\begin{equation}
\label{eq:qY}
\begin{split}
 \lambda(y) = & \sum_{j=1}^\rv  \lambda_j y_j = \sum_{j=1}^\rv \lambda_j  h^* \gat_j(Y) h  
\\   = &  h^* \, \left [ \sum_{j=1}^\rv \lambda_j \gat_j(Y) \right] \, h 
  = h^* q(Y) h.
\end{split}
\end{equation}

If $Y\in\KK(n)$ and $h\in\C^n$ is any unit vector
 then, by Proposition \ref{prop:gammahull}, $y\in\matco(\Gat(\KK))(1).$
Thus,
\[
 0\le \lambda(y) = h^* q(Y) h
\]
and it follows that $q(Y)\succeq 0$. Hence $q(Y)\succeq 0$ on $\KK.$

To prove the converse, suppose $0$ is in the interior of
 $\matco(\Gat(\KK))(1)$ and $q$ is in the real span of
 $\{\gat_1,\dots,\gat_\rv\}.$ Thus there is a $\lambda\in \R^\rv$
 such that $q=\sum_{j=1}^\rv \lambda_j \gat_j.$  View
 $\lambda:\R^\rv\to \R$ as the linear map $\lambda(z)=\sum_{j=1}^\rv \lambda_j z_j$.
 Since $\lambda(0)=0$ and $0$ is in the interior of $\matco(\Gat(\KK)),$
 there exists $y\in\matco(\Gat(\KK))(1)$ such that $\lambda(y)<0.$ Since $y\in\matco(\Gat(\KK))(1),$
 by Proposition \ref{prop:gammahull}
 there exists an $n$, a vector $h\in\C^n$ and $Y\in \KK(n)$ such that 
 $y= h^* \Gat(Y)h$. Thus, by equation \eqref{eq:qY},
\[
 0 > \lambda(y)  = h^* q(Y)  h.
\] 
 Therefore $q(Y)\not\succeq 0$ and the proof is complete.
\end{proof}

\subsection{$\Gat$-Convex polynomials}\label{ssec:bicvxPoly}
A symmetric   $p\in M_\mu(\C\ax)$ is a 
 \df{$\Gat$-convex polynomial}  if, for each 
 $(X,V)\in \Coup$,
\[
  (I_\mu\otimes V)^* p(X)(I_\mu\otimes V) - p(V^*XV) \succeq 0.
\] It is a {\bf $\Gat$-concave polynomial} if $-p$ is $\Gat$-convex.
 
\subsubsection{Free semialgebraic sets} 
\label{sec:semialgsets}
 Given a symmetric polynomial $p\in M_\mu(\C\ax)$ with $p(0)\succ0$ 
  and a positive integer $n$, let 
\[
\wcD_p(n)=\{X\in \gtupn : p(X)\succeq 0\}
\]
   and let $\cD_p(n)$ denote the closure of
\[
  \spos_p(n) =\{X\in \gtupn: p(X)\succ 0\}.
\]
 Let $\wcD_p$ denote the sequence $(\wcD_p(n))_n.$ Likewise
 let $\spos_p = (\spos_p(n))_n$ and $\cD_p=(\cD_p(n))_n.$ The sets $\wcD_p$,
  $\cD_p$, and $\spos_p$ are
 free analogs of basic semialgebraic sets. We refer to all of these (possibly 
 distinct) sets as {\bf free semialgebraic sets}. As an example, the sets $\TV^d$
 of \eqref{e:TV} are free semialgebraic.
 The inequalities arising from signal flow diagrams give rise to free semialgebraic
 sets, or more generally sets defined by rational inequalities.

Free semialgebraic sets that have additional geometric properties, such as being star-like, 
satisfy cleaner versions of our main results, e.g.~Corollary \ref{cor: starseparation}.
One of our main goals along the lines of \cite{HMannals}
is to develop constrained simple representations of semialgebraic sets with certain geometric properties, 
that is, represent them as a positivity set of a $\Gat$-pencil.

\begin{proposition}
\label{p:concave-convex}
If $p\in M_\mu(\C\ax)$ is a $\Gat$-concave polynomial, then $\wcD_p$ and $\spos_p$ 
are $\Gat$-convex.
\end{proposition}

\begin{proof}
If $X \in \wcD_p$ and $(X,V)$ is a $\Gat$-pair, then 
$(I_\mu\otimes V)^*p(X)(I_\mu\otimes V)\succeq 0$ and, 
since $p$ is $\Gat$-concave,
\[
  p(V^*XV)\succeq (I_\mu\otimes V)^*p(X) (I_\mu\otimes V) \succeq 0.
\]
Therefore $V^*XV \in \wcD_p$ and hence $\wcD_p$ is $\Gat$-convex. 
 The same argument shows $\spos_p$ is $\Gat$-convex.
\end{proof}

An $f\in M_{\nu\times \mu}(\C\ax)$ is a {\bf $\Gat$-concomitant} if
\[(I_\mu\otimes V)^* f(X)(I_\nu\otimes V)=f(V^*XV)
\]
 for every $(X,V)\in \Coup$. 

\begin{corollary}
\label{c:gat-neutral-convex}
 If $f\in M_\mu(\C\ax)$ is a $\Gat$-concomitant, then $\wcD_f$ and $\spos_f$ are $\Gat$-convex.
\end{corollary}

\begin{proof}
 If $f$ is a $\Gat$-concomitant, then $f$ is $\Gat$-concave and hence 
 Proposition \ref{p:concave-convex} applies. 
\end{proof}

\begin{remark}\rm
If  $L$ is a monic $\Gat$-pencil and $f_j\in M_{\mu_j\times \mu}(\C\ax)$ for $j=1,\dots, N$
 are  $\Gat$-concomitant, then
\[
M(x) =  \begin{pmatrix}  I_N & \begin{matrix} f_1 \\ \vdots \\ f_N \end{matrix} \\ 
 \begin{matrix} f_1^* & \dots & f_N^* \end{matrix} & L \end{pmatrix}
\]
 is a $\Gat$-concomitant. Hence  $\cD_M$ is $\Gat$-convex. By taking a Schur complement,
 $\cD_M = \cD_p$, for 
\[
 p = L -\sum f_j^* f_j,
\]
which has the form of a monic $\Gat$-pencil minus a sum of hermitian squares
 of a $\Gat$-concomitant polynomials. 
\qed \end{remark}

Rudimentary classification results 
for partially convex free polynomials
exist in \cite{HHLM}; several classes
of $\Gamma$-convex functions for specific cases %
will be given in the sequel \cite{JKMMP+}.

\section{The Operator Setting}
\label{sec:opset}
In this section, the notion of $\Gat$-convexity
 is extended to tuples of operators. While the matrix case
 is our primary interest,  
apparent defects in the geometry, such as
level sets not ``varying continuously,''
 necessitates 
an appeal to  the penumbral operator case. 
We will see in Subsection \ref{sec:op-to-mat} that tools from the operator
 setting lead to results for matrix $\Gat$-convexity.    The remainder of 
 this section is organized as follows.  Subsection \ref{sec:OCSOT} contains
 preliminary results, including an analog of  Proposition \ref{prop:jp} 
 (see Proposition \ref{prop:opco}) and  the key fact that,
 for a strong operator topology (SOT) closed and bounded free
  set $\oK$, the operator convex
 hull of $\Gat(\oK)$ and the $\Gat$-operator convex hull of $\oK$ 
 are again SOT closed (see Theorem \ref{thm: sotclosed}). 
 Versions of the Effros-Winkler theorem  for operator convex sets
 and operator $\Gat$-convex sets are established in 
 Subsections \ref{sec:opEW} and \ref{sec:gatopEW} respectively. 
 The section concludes  with 
the desired
 applications of operator $\Gat$-convexity
 to matrix $\Gat$-convexity for free semialgebraic sets in
 Subsection \ref{sec:op-to-mat}.

\subsection{Operator  $\Gat$-convexity and the strong operator topology}
\label{sec:OCSOT}
Fix an infinite dimensional separable complex Hilbert space \df{$\cH$}  and let 
\df{$\sa$} denotes the $\gv$-tuples of self-adjoint bounded operators on $\cH$. We equip $\sa$
with the maximum norm $\|X\|=\max\{\|X_1\|,\ldots,\|X_\gv\|\}$. \index{$\|X\|$}
A subset $\oK \subset B(\cH)^\gv$ is a \df{free set} if it is closed under unitary similarity 
and closed under direct sums, where $\cH\oplus \cH \oplus \cdots \oplus \cH$ is
 identified with $\cH.$

By analogy with the matricial theory from Section \ref{sec:generalstuff}, 
let $\Gat=(\gat_1,\dots,\gat_r)$ denote a tuple  of {\it symmetric} free polynomials
 with $\gat_j=x_j$ for $1\le j\le \gv\le \rv$, and let $\PhiG:\sa \to B(\cH)^\rv_{sa}$ denote the
 resulting mapping on self-adjoint operator tuples.
  As before, $(X,V)$ is called a \df{$\Gat$-pair} provided
 $X\in \sa$, $V:\cH \to \cH$ is an isometry, and\looseness=-1
\[
 V^* \PhiG(X)V= \PhiG(V^*XV).
\]
 Let \df{$\oCoup$}  denote the collection of (operator) $\PhiG$-pairs.

\index{$\oK$}
A free set $\oK \subset \sa$ is called \df{operator convex} if whenever $X \in \oK$ and 
$V:\cH\rightarrow \cH$ is an isometry, then $V^*XV \in \oK.$ It is called 
\df{operator $\Gat$-convex} if 
\[
 X\in \oK \, \text{ and } \, (X,V)\in \oCoup \, \implies V^*XV\in \oK.
\] In the special case that $\rv=\gv$ (equivalently $\Gat(x)=x$) 
$\Gat$-convexity reduces to ordinary operator convexity.
The \df{operator $\Gat$-convex hull} of a free set $\oK \subset \sa$ is the intersection of all 
operator $\Gat$-convex sets containing $\oK$ and is denoted {$\gammaopco (\oK).$ 
\index{$\gammaopco$}

It is immediate that Propositions \ref{prop:gammahull} and \ref{prop:jp} have operator analogues.

\begin{proposition}
\label{prop:opco}
If $\oK \subset \sa$ is  a free set, then 
\[
\begin{split}
  &\gammaopco(\oK)=\{V^*XV: X\in \oK, \,(X,V)\in\oCoup\}
  \quad\text{ and}\\
 &\PhiG^{-1}(\opco(\PhiG(\oK))) %
 = \gammaopco(\oK).
 \end{split}
\]
\end{proposition}

We now show that the operator 
convex hull of a bounded SOT-closed free set is again SOT-closed, 
eliminating certain technical difficulties in absence of a Caratheodory-type theorem for $\Gat$-convexity. The proof 
uses a Heine-Borel type compactness principle from operatorial noncommutative function theory, 
which was previously applied by
 \cite{Mancuso}.

 \begin{lemma}[\protect{\cite[Lemma 4.5 and Remark 4.6]{Mancuso}}]
 \label{l:shiftform}
Let $X_n$ be a bounded sequence of operator tuples in  $B(\cH)^{\gv}$. 
Then there exists a sequence $U_n$ of unitary
operators on $\cH$ and a subsequence along which 
$U_n^*X_nU_n$ and $U_n^*$ both converge in SOT.
 \end{lemma}

\begin{theorem}
\label{thm: sotclosed}
Suppose $\oK \subset \sa$ is a free set. If $\oK$ is bounded and  SOT-closed,
then $\opco(\PhiG(\oK))$ and $\gammaopco(\oK)$ are SOT-closed. 
\end{theorem}

\begin{proof}
 By the set equality in Proposition \ref{prop:opco}, it suffices to show
 $\opco(\PhiG(\oK))$ is SOT-closed, since 
 $\Gamma$ is SOT-continuous on bounded sets as it is a free polynomial mapping.

Suppose $Y$ is in the SOT-closure of $\opco(\PhiG(\oK)).$ 
There exist isometries $V_n$ and  $X_n\in \oK$ such that
$V_n^* \PhiG(X_n) V_n\stackrel{SOT}{\to}Y$. 
By Lemma \ref{l:shiftform} applied to $(V_n,X_n)\in B(\cH)^{\gv+1},$
there exist unitaries  $U_n$ such that,
after passing to a subsequence,
 $U_n^* V_n U_n\stackrel{SOT}{\to} V$,
$U_n^* X_nU_n\stackrel{SOT}{\to}X$,
and
$U_n^*\stackrel{SOT}{\to}W$,
where $V,W$ are isometries, and  $X\in \oK$ since $\oK$ is free and SOT-closed.

Since $U_n^* V_n = (U_n^*V_nU_n)U_n^*$, we have  $U_n^*V_n\stackrel{SOT}\to VW$ and {\it a fortiori} that $V_n^*U_n\stackrel{WOT}{\to}W^*V^*$.
Moreover, 
$\PhiG(U_n^*X_nU_n)U_n^* V_n\stackrel{SOT}{\to}\PhiG(X)VW$
since multiplication is SOT-continuous on bounded sets.
Note that 
 if $S_n\stackrel{WOT}\to S$  and $T_n\stackrel{SOT}{\to}T$, then
$S_nT_n\stackrel{WOT}{\to}ST$. Hence
\[
\begin{split}
Y & =  \limsot_n  V_n^* \PhiG(X_n) V_n\\
& = \limsot_n V_n^* U_n \PhiG(U_n^*X_nU_n) U_n^* V_n \\
 & =  \limwot_n  V_n^* U_n [\PhiG(U_n^*X_nU_n)U_n^* V_n] \\
 & =  W^*V^* [\PhiG(X)VW] = (VW)^* \PhiG(X)VW.
\end{split}
\]
Hence 
$
 Y = (VW)^* \PhiG(X) VW \in \opco(\PhiG(\oK)).
 $
\end{proof}

\begin{remark}\rm
Note that, in the context of Theorem \ref{thm: sotclosed},
 as $\opco(\PhiG(\oK))$ is convex and SOT-closed, it is also WOT-closed.
The proof that $\opco(\Gat(\oK))$ is SOT-closed shows in fact that if
$P:\sa \to B(\cH)^\rv_{sa}$ is any free polynomial mapping, then  $\opco(P(\oK))$ is SOT-closed.
\qed \end{remark}

\index{$\wcD_p$} \index{$\spos_p$} \index{$\cD_p$} \index{$\woD_p$}
Given a symmetric noncommutative polynomial
$p\in M_\mu(\C\ax)$,
in addition 
 to $\wcD_p$, $\spos_p$, and $\cD_p$ defined in
 Subsection \ref{sec:semialgsets},
 we consider the following
sets describing its positivity 
on the operator level: 
\[
\begin{split}
\woD_p & = \{X\in \sa : p(X)\succeq 0\}; \\
\sopos_p & = \{X\in \sa: p(X)\succ 0\}; \\
\oD_p & = \overline{\sopos_p}^{SOT}.
\end{split}
\]
We also refer to all of these (possibly distinct) sets as 
\df{free semialgebraic sets}. \index{$\sopos_p$}\index{$\oD_p$}
Observe that, whenever they are bounded, $\oD_p$ and $\woD_p$ are SOT-closed and hence
 Theorem \ref{thm: sotclosed} applies.

\subsection{The Effros-Winkler theorem for operator convex sets}
\label{sec:opEW}
 We now show that a version of the Effros-Winkler theorem holds for bounded, SOT-closed free sets
 in the operator convex case. In Subsection \ref{sec:gatopEW}, we prove a
 version for operator $\Gat$-convex sets.
 Given a monic 
$\Gat$-pencil $L=I_N + \sum A_j\gamma_j$ of  size $N$, or an operator $\Gat$-pencil
 $\widetilde{L}=I_{\widetilde{\cH}}+\sum B_j\gamma_j$ for a separable infinite dimensional
Hilbert space $\widetilde{\cH}$ and $B_j \in B(\widetilde{\cH})_{sa},$
 its \df{evaluation on operator tuples} is defined as
\[
\begin{split}
L(X)&=I_N\otimes I_{\cH}+\sum A_j\otimes \gamma_j(X), \\
\widetilde{L}(X)&=I_{\widetilde{\cH}}\otimes I_{\cH}+\sum B_j\otimes \gamma_j(X)
\end{split}
\] 
 respectively for $X\in \sa$. 

The \df{positivity set}
 (resp. \df{strict positivity set}) of a $\Gat$-pencil $L$ in the operator setting is 
\[
 \woD_L :=\{X \in\sa: L(X)\succeq 0\} \ \ (\text{resp.} \, \sopos_L :=\{X\in \sa: L(X)\succ 0\}).
\]

\begin{theorem}\label{thm:opEW}
Suppose $\oK \subset \sa$ is operator convex, SOT-closed,  and contains $0$.
If $Y\notin \oK,$ then there 
is a positive integer $N$ and a monic linear pencil \[L = I_N  + \sum_{j=1}^\gv A_j x_j,
\]
 where $A_j \in \mathbb S_N(\C),$
 such that $L(\oK)\succeq 0$, but $L(Y)\not\succeq 0$. 

In particular, $\oK$ is the intersection
$\bigcap \woD_L$, where the intersection is over %
monic linear pencils $L$ such that $\woD_L\supset \oK$.

Finally, if $\ell$ is a positive integer, $\cF$ is an $\ell$ 
 dimensional subspace of $\cH$ and $Y=Y^\prime \oplus 0$ with respect to the
 orthogonal decomposition $\cF \oplus \cF^\perp$ of $\cH$, then $L$ can be chosen to have
 size $\ell.$
\end{theorem}

The proof of Theorem \ref{thm:opEW} uses Proposition \ref{prop:mppa} below, which
 in turn uses the following Lemma.

\begin{lemma}\label{lem:contraction}
 If $\oK \subset \sa$  is operator convex and contains $0$, then $\oK$ is closed
 under conjugation by contractions: if $X \in \oK$, $C\in B(\cH)$ and 
 $\Vert C\Vert \leq 1,$ then $C^*XC \in \oK$.

 In particular, if $P\in B(\cH)$ is a projection and $Y\in \oK,$ then
 $PYP = Y^\prime \oplus 0\in \oK.$
\end{lemma}

\begin{proof}
If $X \in\oK$, then 
$Z:=\begin{pmatrix} X & 0 \\ 0 & 0\end{pmatrix}\in\oK.$ 
For a contraction $C,$ define the isometry
\[W:= \begin{pmatrix} C \\ (1-C^*C)^{1/2} \end{pmatrix}.
\] 
Now,  $C^*XC=W^*ZW\in\oK$.
\end{proof}

 To a  free subset $\oK\subset \sa$ we associate a matrix convex set $\oKp$.
  Given a positive integer $n$, let \df{$\oKp(n)$}
 denote the set of tuples $X\in \gtupn$ of the form $V^* YV$, where
 $V:\C^n\to \cH$ is an isometry and $Y\in \oK$.  Note that, by Lemma
 \ref{lem:contraction},  $VV^* \, Y \, VV^* \in \oK$ and hence
 if $Y^\prime \in \oKp$, then $Y^\prime\oplus 0\in \oK.$

\begin{proposition}\label{prop:mppa}
If $\oK\subset\sa$ is operator convex,  SOT-closed and contains $0,$ then 
\begin{enumerate}[\rm (a)]
\item  $\oKp$ is a closed matrix convex set containing $0;$ 
\item \label{i:mppa2} 
$\displaystyle 
\oK=\overline{\{Y^\prime\oplus 0\mid Y^\prime\in\oKp\}}^{SOT}; $
\item \label{i:mppa4}  if $Y\not\in\oK$, then there is an $N$ and
   an isometry $V:\C^N\to\cH$ so that $V^*YV\not\in\oKp(N);$
\item \label{i:mppa3}
 if $Y^\prime\in \bbS_N(\C)^g$ and $Y^\prime\oplus 0 \notin \oK,$ then there is a monic linear pencil $M$
 of size $N$ such that $M(\oK)\succeq 0,$ but $M(Y^\prime)\not\succeq 0.$
\end{enumerate}
\end{proposition}

\begin{proof}
 A routine argument shows $\oKp$ is (a free set and) matrix convex. 
 If $(Y_m)$ is a sequence from $\oKp(n)$ that converges to $Y$, then
 $(Y_m\oplus 0)$ is a sequence from $\oK$ that converges in norm, and hence in SOT,
 to $Y\oplus 0.$ It follows that $Y\in \oKp(n)$ and hence $\oKp$ is closed.

 If $Y^\prime\in \oKp,$ then $Y^\prime\oplus 0\in \oK$ as already noted.
 Conversely,  given $X\in\oK$ and a sequence $P_n$ of finite rank projections onto the 
 first $n$ basis vectors of $\cH$,
 we have $P_nXP_n\stackrel{SOT}{\to}X$, $P_nXP_n \in \oK$ by Lemma \ref{lem:contraction}
 and $P_nXP_n$ has the form $Y_n\oplus 0$ for
 $Y_n\in \oKp(n).$ Item \ref{i:mppa2} follows.\looseness=-1

 Item \ref{i:mppa4} follows from item \ref{i:mppa2}.

 To prove item \ref{i:mppa3},  since $Y^\prime \oplus 0\notin \oK,$ it follows that 
  $Y^\prime \notin\oKp(N).$
  Since $\oKp$ is closed and matrix convex, Theorem \ref{t:EW} implies
  there is a monic linear pencil $M$
  of size $N$ such that $M(\oKp)\succeq 0,$ but $M(Y^\prime)\not\succeq 0.$ 
  By item \ref{i:mppa2} and SOT-continuity of $M,$ it follows that
  $M(\oK)\succeq 0.$
\end{proof}

\begin{proof}[Proof of Theorem~\ref{thm:opEW}]
 By Proposition \ref{prop:mppa}, there is an $N$ and an isometry $V:\C^N\to\cH$
 such that $Y^\prime=V^*YV\notin \oKp(N)$.   Hence $Y^\prime \oplus 0\notin \oK$
 and therefore, by Proposition \ref{prop:mppa}\ref{i:mppa3}, there is a monic linear pencil
 $M$ of size $N$ such that $M(\oK)\succeq 0,$ but $M(Y)=M(Y^\prime)\oplus I \not\succeq 0.$

 To prove the last statement, note that, given the form of $Y$ and the hypotheses
 of the theorem,   $Y \not\in \oK$ if and only if $Y^\prime \not\in \oKp(\ell).$
\end{proof}

\subsection{Hahn-Banach separation for operator $\Gat$-convex sets}
\label{sec:gatopEW}
Combining Theorems \ref{thm: sotclosed} and \ref{thm:opEW} yields 
Theorem \ref{thm:GatOPEW} below. It may be
seen as an improvement of Theorem \ref{thm:jp} 
for bounded,  SOT-closed operator $\Gat$-convex sets since it does not require that
the (operator) convex hull of $\PhiG(\oK)$ be closed a priori.

\begin{theorem}
\label{thm:GatOPEW}
Suppose $\oK \subset \sa$ is SOT-closed, operator $\Gat$-convex, bounded, and that $0 \in \opco(\PhiG(\oK)).$ 
For each $Y \notin \oK,$ there is a positive integer $N$ an a monic linear pencil
\[ 
    M = I_N  + \sum_{j=1}^\rv A_j x_j, 
\]
 of size $N$ such that 
 $M(\opco(\PhiG(\oK)))\succeq 0$, but $M(\PhiG(Y))\not\succeq 0$.
 Thus the $\Gat$-pencil $L =M\circ \Gat = I_N+\sum A_j\gat_j$ is positive semidefinite on $\oK$,
 but $L(Y)\not\succeq 0$.

 In particular, $\oK=\bigcap \woD_L$, where the intersection is over all  monic $\Gat$-pencils
$L$ such that $\woD_L\supset \oK$. %

 Finally, suppose $0 \in \oK$ and that $\Gat(0)=0.$ If $\ell$ is a positive integer, $\cF$ is an $\ell$ 
 dimensional subspace of $\cH$ and $Y=Y^\prime \oplus 0$ with respect to the
 orthogonal decomposition $\cF \oplus \cF^\perp$ of $\cH$, then $L$ and $M$ can be chosen to have
 size $\ell.$
\end{theorem}

\begin{proof}
The assumptions imply $\opco(\PhiG(\oK))$ is a bounded, operator convex set containing 0. 
It is SOT-closed by Theorem \ref{thm: sotclosed}. If $Y\notin \oK=\gammaopco(\oK),$ 
then by the set equality in Proposition \ref{prop:opco}, $\PhiG(Y)\notin \opco(\PhiG(\oK)).$ 
Therefore, by Theorem \ref{thm:opEW}, there exists a positive integer $N$ and a monic linear pencil 
\[
 M = I_N  + \sum_{j=1}^{\rv} A_j x_j
\] 
of size $N$ such that $M\succeq 0$ on $\opco(\PhiG(\oK)),$ 
but $M(\PhiG(Y))\not \succeq 0.$ 

 In the case of the final assertion, 
 $\Gat(Y)=\Gat(Y')\oplus \Gat(0)=\Gat(Y')\oplus 0.$
  Since $\Gat(Y')$ has size $\ell$ and $\Gat(Y)\notin \opco(\PhiG(\oK)),$ 
  the monic $\Gat$-pencil $L$ can be chosen to have size $\ell$ by Theorem  \ref{thm:opEW}. 
\end{proof}

\begin{remark}\rm
In Theorems \ref{thm:opEW} and \ref{thm:GatOPEW}, while the convex sets
 consist of operator tuples, outliers are separated from these sets 
by a monic pencil of {\it finite} size; that is, a matrix pencil. 
In the particular cases described in their final statements,
the theorems assert further control over this finite size.
\qed \end{remark}

\begin{lemma} 
 \label{l:zeroininterior}
   For an SOT-closed operator convex set $\oK\subset \sa$, the following are equivalent. 
\begin{enumerate}[\rm (i)]
 \item \label{i:zeroi}
   $0$ is in the (norm) interior of $\oK$;
 \item \label{i:zeroii}
    there is a constant $C$ such that if $L=I+\sum A_j x_j$ is monic linear
   pencil that is positive semidefinite on $\oK$, then $\|A_j\|<C$ for all $j$;
 \item \label{i:zeroiii}
  there is an $\epsilon>0$ such that $\pm \epsilon e_j \in \oK$, where
  $e_j \in \sa$ has $k$-th entry $0$ if $k\ne j$ and $I$ if $k=j$.
\end{enumerate}

 Further, if $M$ is a monic linear pencil  such that $M(\oK)\succeq 0$, then $M\succ 0$
 on $\intr(\oK)$.
\end{lemma}

\begin{proof}
 The implications \ref{i:zeroi} implies \ref{i:zeroiii} is evident.
 To pass from item \ref{i:zeroiii} to item \ref{i:zeroii}, choose $C=\frac{1}{\epsilon}$.

 Now suppose item \ref{i:zeroi} does not hold; that is $0$ is not in the norm interior of $\oK$.  Given $\delta>0$,
 there is a tuple $Y\notin\KK$ with $\|Y\|<\delta$. By Theorem  \ref{thm:opEW},
 there is a monic linear pencil $L=I+\sum A_j x_j$ of finite size such that $L\succeq 0$ on $\KK$,
 but $L(Y)\not\succeq 0.$ Since $I \not\succeq - \sum A_j \otimes Y_j$ and 
\[
 - \sum A_j \otimes Y_j \preceq  \|\sum A_j\otimes Y_j\| \, I \preceq \sum \|A_j \|\, \|Y_j\|\, I
 \preceq \delta \sum \|A_j\| \, I,
\]
 it follows that there is a $j$ such that $\|A_j\|\ge \frac{1}{\gv\delta}$. Hence 
 item \ref{i:zeroii} does not hold and the proof that items
  \ref{i:zeroi}, \ref{i:zeroii} and \ref{i:zeroiii} are equivalent is complete.

 To complete the proof, suppose $X\in \oK$ and $M(X)\succeq 0$, but $M(X)\not\succ 0$.
 Thus, there is a nonzero vector $h$ such that $\langle M(X)h,h\rangle =0$. Since $M$
 is monic linear, $\langle M(tX)h,h\rangle = \|h\|^2 - t \|h\|^2$. Hence, 
 $\langle M(tX)h,h\rangle <0$ for $t>1$. Thus $X$ is not in the interior 
 of $\oK$.
\end{proof}

In Corollary \ref{c:op-gat-separate} below, the topological notions of 
boundary and interior are with respect to the relative
norm topology on $\sa$. Note that, in the case $\rv=\gv$ (equivalently
 $\Gat(x)=x$)
the condition of equation \eqref{intrcontain} in Corollary \ref{c:op-gat-separate}
is automatically satisfied.

\begin{corollary}
\label{c:op-gat-separate}
Suppose $\oK \subset \sa$ is SOT-closed, operator $\Gat$-convex, bounded, 
contains $0$, and $\PhiG(0)=0.$ Suppose also that $0$ is in the (norm) interior of 
$\opco(\PhiG(\oK))$ and that 
\begin{equation}
 \label{intrcontain}
      \intr(\oK)\subset \PhiG^{-1}(\intr(\opco(\PhiG(\oK)))).
\end{equation}
If $Y$ is in the (norm) boundary of $\oK,$ then there exists a  
monic operator $\Gat$-pencil $L=I+\sum A_j \gat_j$ 
 such that $L(X) \succ 0$ for all $X$ in the interior of $\oK$ and such that
  $L(Y)$  is not bounded below by any positive multiple of the identity.
\end{corollary}

\begin{proof}
There is a sequence $(Y_n)$ from $\sa$ such that $Y_n \notin \oK$
 and $(Y_n)\rightarrow Y$ in norm. By the first statement in Theorem \ref{thm:GatOPEW},
 there are monic linear pencils $M_n$ of finite size such that $M_n(\opco(\Gat(\oK)))\succeq 0$,
 but $M_n(\Gat(Y_n))\not\succeq 0.$  Since $0$ is in the (norm) interior of $\opco(\PhiG(\oK))$,
 Lemma \ref{l:zeroininterior} implies $M:=\bigoplus_n M_n$ defines a monic pencil whose coefficients
 are bounded operators on a separable Hilbert space. Put $L:=M\circ \Gat.$

 By construction, $M\succeq 0$ on
  $\opco(\PhiG(\oK))$. 
  By Lemma \ref{l:zeroininterior} applied to   $\opco(\PhiG(\oK))$,
  it follows that   $M(Z)\succ 0$
  for $Z\in \intr(\opco(\PhiG(\oK))).$ By the inclusion in (\ref{intrcontain}), if 
  $X\in \intr(\oK)$, then $\Gat(X) \in \intr(\opco(\PhiG(\oK)))$, 
  and $M(\Gat(X))=L(X)\succ 0.$

  Finally, we show  $L(Y)$ is not bounded below by a positive multiple of the identity.
  Since the sequence $Y_n$ tends to $Y$ in norm, $\Gat(Y_n)\rightarrow \Gat(Y)$
  in norm.
As $M_n(\Gat(Y_n))$ is a summand of  $M(\Gat(Y_n))$, it follows that 
$M(\Gat(Y_n))\not \succeq 0$.
Since $M(\Gat(Y_n))$ converges to $M(\Gat(Y))$ in norm,
 there does not exist an $\epsilon>0$ such that 
 $L(Y)=M(\Gat(Y))\succeq \epsilon$.
\end{proof}

\begin{remark}\rm
 In the case of $y^2$-convexity and 
 without hypotheses that guarantee a uniform bound on the coefficients,
 it is not clear how to avoid, in the proof of Corollary \ref{c:op-gat-separate}
 and after scaling, having the pencils $L_n = M_n\circ \Gat$ converge to
 a $\Gat$-pencil $L$ that is not equivalent to a monic $\Gat$-pencil,  such as
\[
 L(x,y) = \begin{pmatrix} 1 & y \\ y & y^2 \end{pmatrix},
\]
 which is always positive semidefinite, but never positive definite.
 We view this as a symptom of the fact that the image of $\Gamma$ in
 this case lies in the boundary of its operator convex hull.
 The same issue arises in describing operator $\Gat$-convex sets defined by a single
 (operator-valued) monic $\Gat$-pencil. 
 \qed
 \end{remark}

\subsection{Applications of Theorem \ref{thm:GatOPEW} to matricial $\Gat$-convexity}
\label{sec:op-to-mat}
We say $p\in M_\mu(\C\ax)$ is \df{\indifferent} 
if $\cD_p=\wcD_p$ and
 the set $\woD_p$ is bounded and $\oD_p=\woD_p$.

It is clear that 
 $\cD_p\subset \wcD_p$ and 
$\oD_p \subset \woD_p$.
However equality may not hold: consider  $p(x)=-(1-x^2)^2$.
An example of a \indifferent polynomial is $p(x,y)=1-x^2-y^{2d}$.

A simple and natural geometric sufficient condition for
 regularity may be described as follows. 
We say $p\in M_\mu(\C\ax)$ is \df{star-like} if the set $\woD_p$ is bounded
and if  $X\in \sa$ and $p(X)\succeq 0,$ 
then $p(tX)\succ 0$ for all $0\le t<1.$ 
If $p$ is star-like and
$\woD_p \neq \varnothing$, then
 $p(0)\succ 0.$

\begin{example}\label{1-x^2-y^{2d}star}\rm
 For $d$ a positive integer, $p(x,y)=1-x^2-y^{2d}$ is star-like since
\[
p(t(x,y))= (1-t^2)+t^2 p(x,y) + t^2(1-t^{2d-2})y^{2d}.
\] Thus, if $p(X,Y)\succeq 0,$ then $p(t(X,Y))\succeq 1-t^2$ for $0\leq t<1.$
\end{example}

\begin{proposition}
\label{prop:star}
If $p\in M_\mu(\C\ax)$ is star-like, then $p$ is a \indifferent polynomial.
\end{proposition}

\begin{proof}
Clearly, $t\woD_p\subset\intr\oD_p$ for $0<t<1$, 
and $X=\lim_{t\nearrow1}tX$. Similarly, the matrix
case holds.
\end{proof}

For a  \indifferent polynomial $p$, Corollary \ref{cor: starseparation} --
 a separation result for $\cD_p$ under the assumption
 that $\oD_p$ is operator $\Gat$-convex --
 is an immediate consequence of Theorem \ref{thm:GatOPEW}.

\begin{corollary}\label{cor: starseparation}
Let $p\in M_\mu(\C\ax)$ be a  
 \indifferent polynomial such that $\oD_p$ is operator $\Gat$-convex and bounded with
$0\in \cD_p$ and suppose $\PhiG(0)=0.$ If $Y\in \gtupl$ and $Y \not \in \cD_p(\ell),$ then there is
a monic $\Gat$-pencil $L$ of size $\ell$ such that $L\succeq 0$ on $\cD_p$ but $L(Y)\not\succeq 0.$
 In particular, $\cD_p=\cap \wcD_L,$ where the intersection is over
 all  monic $\Gat$-pencils $L$ such that $L(\cD_p)\succeq 0$. 
\end{corollary}

\begin{example}\rm
\label{eg:TVd}
In light of Example \ref{1-x^2-y^{2d}star} and Corollary \ref{cor: starseparation}, if $p(x,y)=1-x^2-y^{2d}$ and 
$(X,Y)\not \in \cD_p,$ then there is a monic $y^2$-pencil $L$ such that $L\succeq 0$ on $\cD_p$ and such that
$L(X,Y)\not\succeq 0.$ 
 Hence $\cD_p =\cap \wcD_L$, where the intersection is over all
 monic  $y^2$-pencils $L$ that are positive semidefinite on $\cD_p$.
  This example is explored further in Proposition \ref{p:TVd}.
\end{example}

If $p$ is \indifferent with $p(0)=1$ and $\oD_p$ is operator $xy$-convex 
(see Example \ref{eg:xyconvex}), then Corollary \ref{cor: starseparation} 
says $\cD_p$ arises from Bilinear Matrix Inequalities (BMIs). In this case more can be said.
\index{Bilinear Matrix Inequalities}\index{BMI}

\begin{proposition}
 \label{p:BMI}
    Suppose $p\in M_\mu(\C\ax)$ is \indifferent, $p(0)\succ 0$, and  $\oD_p$ is operator
 $xy$-convex. If  $\ell$ is a positive integer and  $Y$ is in the boundary
  of $\cD_p(\ell)$, then there is a monic $xy$-pencil $L$ of size $\ell$ such that $L(X)\succ 0$
  for $X$ in the interior of $\cD_p$, but $L(Y)\not \succ 0$.
\end{proposition}

\begin{lemma}
 \label{l:preBMI1}
 If $p \in M_\mu(\C\ax)$ and  $p(0)\succ 0,$ then there is a constant $\beta$
 such that if 
\[
 L(x,y) = I + Ax+By + Cxy+ C^*yx,
\]
 is a monic $xy$-pencil that is positive semidefinite
 on $\cD_p$, then 
\[
 \|A\|, \, \|B\|, \, \|C+C^*\|, \, \|C-C^*\| \le \beta.
\]
\end{lemma}

\begin{proof}
 Since $p(0)$ is positive, there is an $\epsilon>0$ such that
$(\pm \epsilon,0),(0,\pm \epsilon), \pm (\epsilon,-\epsilon)\in \cD_p(1)$ and
$\pm (X,Y)\in\cD_p(2),$ where 
\[
 (X,Y) = \left ( \begin{pmatrix} 0 & \epsilon \\ \epsilon & 0 \end{pmatrix}, 
  \begin{pmatrix} \epsilon & 0 \\ 0 & \epsilon \end{pmatrix} \right ),
\]
 It follows that $\|A\|,\, \|B\| \le \frac{1}{\epsilon}$.
 Likewise, 
\[
0\preceq  L(\epsilon,\epsilon)+L(-\epsilon,-\epsilon) = 2I + \epsilon^2 (C+C^*)
\]
and
\[
0\preceq L(\epsilon,-\epsilon) + L(-\epsilon,\epsilon)= 2I - \epsilon^2 (C+C^*).
\]
Thus $\|C+C^*\|\le \frac{2}{\epsilon^2}.$
Further, 
\[
0\preceq  L(X,Y) + L(-(X,Y)) = 2I + \epsilon^2 (C-C^*)\begin{pmatrix} 0 & 1\\-1 & 0 \end{pmatrix}.
\]
Hence $\|C-C^*\|\le \frac{2}{\epsilon^2}$.  Thus, there is a constant  $\beta$
 such that if $L$ is a monic $xy$-pencil
and $L$ is positive semidefinite on $\cD_p$, then the coefficients of $L$ are
 all bounded (in norm) by $\beta$.
\end{proof}

\begin{lemma}
 \label{l:preBMI2}
 Suppose $S\subset \mathbb S_N(\C)^2.$ If $L$ is a monic 
 $xy$-pencil that is positive semidefinite on $S$, then
 $L$ is positive definite on the interior of $S$.
\end{lemma}

\begin{proof} 
 Suppose $L$ is a monic $xy$-pencil,
\[
 L(x,y) = I + Ax+By + Cxy+ C^*yx,
\]
 that is positive semidefinite on $S.$
 Arguing by contradiction, suppose  $(X,Y)$ is in the interior of $S,$
 but $L(X,Y)\not \succ 0$. Hence there is a vector $h$ such that $\|h\|=1$ and
  $L(X,Y)h=0.$  Set 
\[
\begin{split}
 q_1(t)= & \langle L(tX,tY)h,h\rangle
\\  = & 1  + t\langle [A\otimes X+B\otimes Y]h,h\rangle
  + t^2  \langle [ C\otimes XY + C^*\otimes YX]h,h\rangle.
\end{split}
\]
 Thus $q_1$ is quadratic, $q_1(0)=1,$ $q_1(1)=0$, and $q_1(t)\ge 0$ for $t$ near $1$. Hence
 the coefficient of $t^2$ is positive; that is
\[
 \alpha:=  \langle [C\otimes XY + C^*\otimes YX]h,h\rangle > 0.
\]

Let 
\[
\begin{split}
q_2(t) = & \langle L(tX,Y)h,h\rangle 
\\  = & \langle L(X,Y)h,h\rangle  + 
     (t-1) [\langle A\otimes X \, h,h\rangle  \, + \, \alpha]
\\ = &  (t-1) [\langle A\otimes X \, h,h\rangle  \, + \, \alpha].
\end{split}
\]
Since $q_2(t)\ge 0$ for $t$ real and near $1$, 
\begin{equation}
 \label{e:xyA}
 \langle A\otimes X \, h,h\rangle  \, + \, \alpha = 0
\end{equation}
 A similar argument shows,
\begin{equation}
 \label{e:xyB}
  \langle B\otimes Y \, h,h\rangle \, +\, \alpha =0.
\end{equation}
Combining equations \eqref{e:xyA} and \eqref{e:xyB} gives,
\[
 - \langle A\otimes X h,h\rangle = - \langle B\otimes Y h,h\rangle 
 = \alpha>0.
\] 
 Since $\langle L(X,Y)h,h\rangle =0$, it follows that $\alpha=1$.
 Hence, 
\[
   q_3(t) := t \, \langle L(tX,\frac{1}{t}Y)h,h\rangle  
  = 2t   -t^2 -1 = -(t-1)^2. 
\]
 On the other hand, $(tX,\frac{1}{t}Y)\in S$ for $t$ real and near $1$
 and hence $q_3(t)\ge 0$ for such $t$ and we have reached a contradiction.
 Thus, $L(X,Y)\succ 0$.
\end{proof}

\begin{proof}[Proof of Proposition~\ref{p:BMI}]
 Suppose $Y$ is in the boundary of $\cD_p(\ell)$. Thus, there is a sequence 
 $(Y_n)$ 
  converging to $Y$ with each $Y_n\not\in \cD_p(\ell)$. 
 By Corollary \ref{cor: starseparation}, there exists monic $xy$-pencils
 $L_n$ of size $\ell$  
 such that $L_n$ is positive semidefinite on  $\cD_p$ 
 and $L_n(Y_n)\not \succeq 0$.  Since the coefficients of $L_n$ all have size $\ell$
 and are, by Lemma \ref{l:preBMI1}, uniformly bounded, by passing to a subsequence if necessary, 
 we may assume $L_n$ converges (coefficient-wise) to a  monic $xy$-pencil $L$
 of size $\ell.$
 Thus $L$ is positive semidefinite on $\cD_p$ and $L_n(Y_n)$ converges to $L(Y)$.
Hence $L(Y)\not \succ 0$  and, since $L$ is monic and $L(\cD_p)\succeq 0$,
 Lemma \ref{l:preBMI2} implies $L\succ 0$ on the interior of $\cD_p$
 and the proof is complete. 
\end{proof}

\section{$y^2$-Convex Sets}%
\label{sec:y2convex}
This section
treats $y^2$-convex sets, where it is shown that a  free set 
$\KK \subset \gtup\times\htup$ is $y^2$-convex if and only if, for
each $n$  and $Y\in \htupn$, the \df{slice} $\{X: (X,Y)\in \KK(n)\}$ is convex
in the ordinary sense. It is also shown that the
$y^2$-convex sets $\TV^d$ of equation \eqref{e:TV} are the positivity set
of a single (finite) $y^2$-pencil. 
By comparison, as
noted in Example \ref{eg:TVd}, the general theory only guarantees that 
$\TV^d=\cD_p$ is the intersection of, possibly infinitely many, 
positivity sets of (finite) monic $y^2$-pencils.

Suppose  $\cS\subset \gtup\times\htup$  and write elements $Z$ of $\cS(n)$ as $Z=(X,Y)$ with 
$X\in \gtupn$ and $Y\in \htupn$. 
In the case $\cS$ is free, it is called \df{convex in $x$} if, for each $n$ and $Y\in\htupn$, 
the slice
\[
 \cS[Y]:=\{X\in \gtupn: (X,Y)\in \cS(n)\} \subset \gtupn,
\]
is convex (in the usual sense as a subset of $\gtupn$). In this setting  $y^2$-convex
means $\Gat$-convex for
\[
 \Gat=\{x_1,\dots,x_\gv,y_1,\dots,y_\hv,y_1^2,\dots,y_\hv^2\}.
\]

\begin{proposition}
\label{prop:y2convex=convexinx}
A free set  $\cS\subset \gtup\times \htup$ is $y^2$-convex if and only if  it is convex in $x$.
\end{proposition}

\begin{proof}
 Suppose $\cS$ is a $y^2$-convex free set.
 Fix $n$ and $E\in \htupn$. Given $A,B\in \gtupn$ such that $(A,E),(B,E)\in \cS(n)$,
 observe that, since $\cS$ is a free set
\[
  (X,Y) := (A,E)\oplus (B,E) = \big ( \begin{pmatrix} A & 0\\0& B \end{pmatrix},
  \, \begin{pmatrix} E & 0\\0& E\end{pmatrix}  \big ) \in \cS(2n).
\]
 Since $\cS$ is $y^2$-convex and, for $0\le t \le 1$,
 $V=\begin{pmatrix} \sqrt{t} I_n & \sqrt{1-t} I_n\end{pmatrix}^*$
is an isometry satisfying $V^* Y^2V= (V^*YV)^2$ (so that $((X,Y),V)$ is a $y^2$-pair),
\[
 V^* (X,Y)V = (tA+(1-t) B,E)  \in \cS(n).
\]
 Thus $\cS$ is convex in $x$.

 Now suppose $\cS$ is a free set that is convex in $x$. To prove $\cS$ is $y^2$-convex,
 suppose $(X,Y)\in\cS(n+m)$ and $V:\C^n\to \C^{n+m}$ is an isometry such that
 $((X,Y),V)$ is a $y^2$-pair. Explicitly $V^* Y^2 V= (V^*YV)^2$ and thus the 
 range of $V$ reduces $Y$. Hence, with respect to the direct sum $\C^n\oplus \C^m$,
\[
 (X,Y) = \left ( \begin{pmatrix} X_{11} & X_{12} \\ X_{12}^* & X_{22} \end{pmatrix},
 \,  \begin{pmatrix} Y_{11} & 0 \\ 0 & Y_{22} \end{pmatrix} \right ).
\]
 Letting $U$ denote the unitary matrix,
\[ 
 U =\begin{pmatrix} I_n & 0\\ 0 & - I_m \end{pmatrix},
\]
 $U^*(X,Y)U\in \cS(n+m)$ since free sets are closed under unitary similarity. Since
 $\cS$ is convex in $x$, the slice $\cS[Y]$ is convex and thus
\[
 (X^\prime,Y^\prime):= \frac 12 [ U^*(X,Y)U+(X,Y)] = 
 \left ( \begin{pmatrix} X_{11} & 0 \\0 & X_{22}\end{pmatrix}, \,
 \begin{pmatrix} Y_{11} & 0 \\ 0 & Y_{22} \end{pmatrix} \right )  \in \cS(n+m).
\]
 Finally, since $\C^n\oplus \{0\}$ reduces $(X^\prime,Y^\prime)$ and free sets
are closed with respect to restrictions to reducing subspaces, $(X_{11},Y_{11})\in \cS(n)$
and hence $\cS$ is $y^2$-convex. 
\end{proof}

 A fundamental question is:
if $\KK$ is free semialgebraic and $\Gat$-convex, then is $\KK$ 
the positivity set of (a) a  $\Gat$-concomitant or, more restrictively,
(b) a $\Gat$-pencil?
Since, for positive integers $d$,
the symmetric polynomial $x^2+y^{2d}$ is $y^2$-convex, $p=1-x^2-y^{2d} $ is
$y^2$-concave and hence, by Proposition \ref{p:concave-convex}, 
 $\wcD_p=\TV^d$ is $y^2$-convex.

\begin{proposition}
\label{p:TVd}
   For $d\in\mathbb N$ and  $p_d=1-x^2-y^{2d},$   there is a monic $y^2$-pencil $L_d$ of size $d+1$ such that
 $\wcD_{L_d} =\wcD_{p_d}$.
\end{proposition}

\begin{proof}
 The pencils $L_1$ and $L_2$ are trivial to construct. For $d=3,4$ one can take
\[
\begin{split}
 L_3(x,y) & =  \begin{pmatrix} 1 & 0 & 0& x \\ 0& 1 & y & y^2 \\ 0 & y & 1+y^2 & \frac{y}{2} \\
   x & y^2 & \frac{y}{2} & 1+\frac{y^2}{4} \end{pmatrix}, \\
 L_4(x,y) & =  \left(
\begin{array}{ccccc}
 1 & 0 & 0 & 0 & x \\
 0 & 1 & y & 0 & y^2 \\
 0 & y & y^2+1 & y & 0 \\
 0 & 0 & y & y^2+1 & \frac{1}{8} \left(-4 y^2-5\right) \\
 x & y^2 & 0 & \frac{1}{8} \left(-4 y^2-5\right) & \frac{5 y^2}{8}+\frac{89}{64} \\
\end{array}
\right).
\end{split}
\]
That $L_4$ is not monic is
easily remedied since its constant term is positive definite.

   A recipe for constructing such
pencils is the following. 
Fix $d$. For $0\le k\le d-2,$ 
set 
\[
\alpha_k=\sqrt{\frac{d-1-k}{d-1}}, \quad
c_0=1 \text{  and  }
c_k=\frac{\alpha_k}{\alpha_{k-1}}
\text{ if }k>0.
\]
Then $ \prod_{j=1}^k c_j = \alpha_k$.
Let $q=\sqrt{d-1}(y^2-1)$ and let
\[
  W = \begin{pmatrix} 1 & 0& 0&0 &\dots & 0\\
                     -c_1y& 1&0&0& \dots & 0 \\
                     0&-c_2 y&1&0 &\dots &0\\
                   \vdots &\ddots& \ddots & \ddots& \ddots & \vdots \\
        \vdots &\ddots& \ddots & \ddots& \ddots & \vdots \\
                     0 & 0 &0&\cdots & -c_{d-2} y & 1  \\
                  \alpha_0 q   & \alpha_1 y q & \alpha_2 y^2 q & \dots &\dots & \alpha_{d-2} y^{d-2} q
\end{pmatrix} \in M_{d\times (d-1)}(\C).
\]
Then
\[
 WW^* =\begin{pmatrix} 1 & -c_1y & 0 & 0 &0 &\dots  & q \\
            -c_1y & 1+c_1^2y^2 & -c_2 y & 0 &0 & \dots & 0 \\
            0 & -c_2y & 1 +c_2^2 y^2 & -c_3 y &0 &\dots &0 \\
 \vdots &\ddots& \ddots & \ddots & \ddots& \ddots & \vdots \\
 0 &0 & \ddots & \ddots  & 1+c_{d-3}^2 y^2 & -c_{d-2}y & 0 \\
  0&0&0&\ddots&\cdots & 1+c_{d-2}^2 y^2 & 0 \\
  q&0&0&0&\cdots&0 & q^2( \sum_{j=0}^{d-2} \alpha_j^2 y^{2j}) \end{pmatrix}\in M_d(\C).            
\]
 Next observe, for $0\le k \le d-4,$
\[
\begin{split}
 0 &=  \alpha_k^2  - 2\alpha_{k+1}^2 + \alpha_{k+2}^2 \\
 0 &= \alpha_{d-3}^2 -2 \alpha_{d-2}^2.
\end{split}
\]
Hence
\[
 q^2( \sum_{j=0}^{d-2} \alpha_j^2 y^{2j}) = (d-1) -d y^2 + y^{2d}.
\]
Let 
\[
 M = WW^* + \begin{pmatrix} 0_{d-1,d-1} & 0_{1,d-1}\\ 0_{d-1,1} & 1-y^{2d} \end{pmatrix}.
\]
 Finally, set
\[
 L_d =\begin{pmatrix} 1 & \begin{pmatrix} 0&0&\dots & x \end{pmatrix} \\
           \begin{pmatrix} 0\\0\\ \vdots \\ x \end{pmatrix} &  M \end{pmatrix}.
\]
Now $L_d$ is not monic 
but its constant term is positive definite,
 so a simple scaling produces an equivalent monic linear
 pencil.
\end{proof}

\end{document}